\documentclass[11 pt]{article}
\usepackage{amssymb}
\usepackage{amsmath}
\usepackage{amsthm}
\usepackage{amscd}
\usepackage{graphicx}
\usepackage{hyperref}
\usepackage[all]{xy}
\usepackage{graphicx}

\usepackage{color}

\usepackage{amsfonts}
\usepackage{amssymb}
\usepackage{amsmath}
\usepackage{amsthm}
\usepackage{amscd}
\usepackage{pxfonts}
\usepackage{stmaryrd}
\usepackage{amsfonts}
\usepackage{bbm}
\usepackage{amscd}
\usepackage{mathrsfs}
\usepackage{latexsym,amsxtra}
\usepackage[T1]{fontenc}
\usepackage{lmodern}
\usepackage[all]{xy}
\usepackage{nicefrac,mathtools}
\usepackage{microtype}

\usepackage{enumitem}  
\usepackage{calc}
\renewcommand{\theequation}{\arabic{section}.\arabic{equation}}
\def\vbar{\mathchoice{\vrule height6.3ptdepth-.5ptwidth.8pt\kern- .8pt}
{\vrule height6.3ptdepth-.5ptwidth.8pt\kern-.8pt} {\vrule
height4.1ptdepth-.35ptwidth.6pt\kern-.6pt} {\vrule
height3.1ptdepth-.25ptwidth.5pt\kern-.5pt}}
\def\<{\langle}
\def\>{\rangle}
\def\a{\alpha}
\def\b{\beta}

\def\c{\cdot}

\def\r{\rho}

\newtheorem{df}{Definition}[section]
\newtheorem{thm}{Theorem}[section]
\newtheorem{cor}{Corollary}[section]

\newtheorem{prop}{Proposition}[section]
\newtheorem{ex}{Example}[section]

\allowdisplaybreaks

\date{}
\begin{document}

\title{ \textbf{skew-supersymmetric solution of the  super Malcev  Yang-Baxter equation and Pre-Malcev superalgebras }}
\author{ F. Harrathi}
\author{{ Fattoum Harrathi\footnote {  E-mail: harrathifattoum285@gmail.com} }\\{}\\
{  University of Sfax, Faculty of Sciences Sfax,  BP
1171, 3038 Sfax, Tunisia}}
 \maketitle
\noindent\hrulefill

\noindent {\bf Abstract.}
The purpose of this paper is to introduce the notion of pre-Malcev superalgebras as the algebraic structure behind the super $\mathcal{O}$-operators on Malcev superalgebras.  Moreover, the relations among
Malcev superalgebras, pre-Malcev superalgebras and pre-alternative superalgebras
are established. Then, we study the operator forms of the classical Yang-Baxter equation (CYBE)
in Malcev superalgebras  and give their relationship
with super $\mathcal{O}$-operators. There are close relationships between the
CYBE in Malcev superalgebras and pre-Malcev superalgebras which can be interpreted
through the super $\mathcal{O}$-operators.

\noindent \hrulefill


 \normalsize\vskip0.5 cm


\tableofcontents
\section{
Introduction
}

 As a promotion of Lie algebra, the notion of Malcev algebras  were introduced in 1955  by A. I. Malcev \cite{Malcev55:anltcloops}, who called these objects Moufang-Lie algebras because of their connection with analytic Moufang loops.
Motivated by the study
of physics and the geometry of smooth loops, Malcev algebras have been widely used.
Similarly to the tangent algebra of a Lie group is a Lie algebra, the tangent algebra of
a locally analytic Moufang loop is a Malcev algebra, see\cite{kerdman}.  A Malcev algebra is a non-associative algebra with a skew-symmetric multiplication that satisfies the Malcev identity. That is to say, a vector space $A$ with an anti-symmetric multiplication $[\c,\c]$ that satisfies, for all $x,y,z \in A$, the Malcev identity
\begin{equation}
\label{malcev}
J(x,y,[x,z]) = [J(x,y,z),x],
\end{equation}
where $J(x,y,z) = [[x,y],z] + [[z,x],y] + [[y,z],x]$ is the Jacobian \cite{Sagle}. Other influential work includes \cite{Abd el Malek,Albuquerque,kuzmin68:Malcevalgrepr,Mahaligesshwara}.

Pre-Malcev algebras have been studied extensively since \cite{Madariaga} which are the generalization of
pre-Lie algebras, in the sense that any pre-Lie algebra is a pre-Malcev algebra but the converse
is not true. Studying pre-Malcev algebras independently is significant not only to its own further
development, but also to develop the areas closely connected with such algebras. A pre-Malcev algebra is a vector space $A$ endowed with a bilinear
product $\cdot$ satisfying the following identity
 for $x, y, z, t \in A$,
 \begin{equation}\label{PMALG}
 [y, z] \cdot (x \cdot t)+ [[x, y], z] \cdot t+ y \cdot ([x, z] \cdot t)- x \cdot (y \cdot (z \cdot t)) + z \cdot (x \cdot (y \cdot t))=0,
\end{equation}
where $[x,y]=x\cdot y-y\cdot x.$

The existence of sub-adjacent Malcev algebras and compatible pre-Malcev algebras was given in \cite[Proposition 5]{Madariaga}. For a given pre-Malcev algebra $(A, \cdot)$, there is
a Malcev algebra $A^{C}$ defined by the commutator $[x, y] = x \cdot y - y\cdot x$, and the left multiplication
operator in $A$ induces a representation of Malcev algebra $A^{C}$.

In paper \cite{Berezin}, as Lie algebras of some generalized groups, Lie superalgebras are introduced.
In 1977, Kac deal as applications, Li prove that certain lowest weight modules for
some well-known infinite-dimensional Lie algebras or Lie superalgebras have natural vertex
operator superalgebra structures in \cite{Li}. Then as a natural generalization of Malcev
algebras and Lie superalgebras, the theory of Malcev superalgebras has quite developed.
In particular, Lie superalgebras are examples of Malcev superalgebras. And the definitions
and basic facts of the theory of Malcev superalgebras can be found in  \cite{Albuquerque1,Shestakov}.

Classical Yang-Baxter equation (CYBE) first arose in the study of inverse scattering
theory (see \cite{Faddeev,Faddeev1}). It has a profound connection with many branches of
mathematical physics and pure mathematics (\cite{BelavinDrinfeld}, etc). In particular, CYBE can
be regarded as a “classical limit” of quantum Yang-Baxter equation, which
plays an important role in the study of classical integrable system (\cite{Semenov}). The
generalization of the ordinary CYBE in the super case, namely, the CYBE in Lie
superalgebras, or the graded CYBE (\cite{Gould,Zhang}), or the super CYBE (\cite{Andruskiewitsch}), has been
studied widely since 1990s. One of the motivations to study such a structure is
its relationship with the solutions of quantum Yang-Baxter equation associated to quantum
supergroups or quantum superalgebras or quantized universal enveloping
superalgebras (\cite{Bracken,Floreanini,Yamane,Yamane1}). In fact, many important results on this topic have already
obtained like the close relationships between the CYBE in Lie superalgebras and
Lie bi-superalgebras (\cite{Gould}), Poisson-Lie supergroups (\cite{Andruskiewitsch}) and so on.

In this paper, super Malcev Yang-Baxter equation (Super MYBE) in Malcev superalgebras and super $\mathcal{O}$-operators are
introduced. We may concentrate on the case that $|r| = 0 $ (i.e., $r$ is even). It is hard to
study the case when $r$ is odd. Moreover, we exploit pre-Malcev superalgebras as super-versions of  pre-Malcev
algebras and study their relations with
Malcev superalgebras and pre-alternative superalgebras. The paper is organized as follows. In Section 2,
we recall some basic facts on super vector spaces and  Malcev superalgebras, then present some fundamental results on representations
of Malcev superalgebras, especially focusing on adjoint
and co-adjoint representations.  In Section 3,
we introduce the  super $\mathcal{O}$-operators and pre-Malcev superalgebras and give their relations with super $\mathcal{O}$-operators, Malcev superalgebras and pre-alternative superalgebras. In Section 4, we  construct
a direct relation between super $\mathcal{O}$-operator and super Malcev Yang-Baxter equation and
show that pre-Malcev algebras are the natural underlying structures.

\renewcommand{\theequation}{\thesection.\arabic{equation}}
In this paper, we  consider all vector space are 
finite dimensional   over an algebraically closed commutative
field $\mathbb{K}$ of characteristic zero.
\section{Some basic properties of Malcev superalgebras}
 In this section we presents fundamental concepts and develop some helpful results that we will use
later. 
Let $V$ be a vector space over a field $\mathbb{K}$. The space $V$ is called a super (i.e.,
$\mathbb{Z}_{2}$-graded) vector space if $V = V _{\bar{0}}\oplus V_{\bar{1}}$.
The elements in $V_{\bar{0}} \cup V_{\bar{1}}$ are called homogeneous. Denoted by $\mathcal{H}(V)$ the set of homogenous elements of $V$. We use the
expression $|x|$ to denote the parity index of the homogeneous element $x$, where
$$|x| =\left\{
         \begin{array}{ll}
           0, & \hbox{$x\in V_{\bar{0}}$ ;} \\
           1, & \hbox{$x\in V_{\bar{1}}$.}
         \end{array}
       \right.
$$
It is assumed in this paper that all elements are the homogeneous of their
corresponding super vector spaces.

\begin{df}
A left alternative superalgebra is a $\mathbb{Z}_2$-graded vector space $A=A_{\bar{0}}\oplus A_{\bar{1}}$ equipped with a bilinear product $\star:A\times A\to A$  obeying $A_{\alpha}\star A_{\beta} \subset A_{\alpha+\beta}$, $\forall \alpha,\beta \in\mathbb{Z}_2$ and the left alternative super identity,
\begin{align}
&as(x, y, z)+(-1)^{|x||y|}as(y, x, z)=0,  \quad \forall \ x, y, z\in A.    \label{lefthomalternativesuperidentity}
\end{align}
where $as(x, y, z)=(x\star y) \star z-x\star (y\star z)$ is the associator.
\end{df}
\begin{df}
A right alternative superalgebra is a $\mathbb{Z}_2$-graded vector space $A=A_{\bar{0}}\oplus A_{\bar{1}}$ equipped with a bilinear product $\star:A\times A\to A$  obeying $A_{\alpha}\star A_{\beta} \subset A_{\alpha+\beta}$, $\forall \alpha,\beta \in\mathbb{Z}_2$ and the right alternative super identity
\begin{align}
&as(x, y, z)+(-1)^{|y||z|}as(x, z, y)=0,  \quad \forall \ x, y, z\in A.   \label{righthomalternativesuperidentity}
\end{align}
\end{df}

\begin{df}
An alternative superalgebra is one which is both left and right alternative superalgebra.
\end{df}

\begin{df}
A Malcev superalgebra is a $\mathbb{Z}_2$-graded vector space $A = A_{\bar{0}}\oplus A_{\bar{1}}$ over a field $\mathbb{K}$ equipped
with a $\mathbb{K}$-bilinear map $[\c,\c] : A\times A \longrightarrow A$ satisfying $[A_{\alpha}, A_{\beta}] \subset A_{\alpha+\beta}$, $\forall \alpha,\beta \in\mathbb{Z}_2$ and
the following conditions
\begin{enumerate}
 \item[(i)] $[x,y]=-(-1)^{|x||y|} [y, x],$
  \item[(ii)]
$(-1)^{|y||z|}[[x, z], [y, t]] = [[[x, y], z], t] +(-1)^{|x|(|y|+|z|+|t|)} [[[y, z], t], x]$\\
$\qquad\qquad + (-1)^{(|x|+|y|)(|z|+|t|)} [[[z, t], x], y] + (-1)^{|t|(|x|+|y|+|z| )} [[[t, x], y], z],$
\end{enumerate}
for all $x, y, z,t \in \mathcal{H}(A)$. An element of $A_{\bar{0}}$ is called even and an element of $A_{\bar{1}}$ is called odd.
 \label{df:malcevsuperalg}
\end{df}
\begin{thm}\label{thm:alttomalc}
Let  $(A,\star)$ be  an alternative superalgebra. Then $(A, [ \c,\c ])$ is a Malcev admissible superalgebra, where
\begin{align*}
[x, y] = x\star y - (-1)^{|x||y|}y\star x,\quad \forall x,y\in \mathcal{H}(A).
\end{align*}
\end{thm}
Let $A$ be any superalgebra and $V$ be a $\mathbb{Z}_2$-graded vector space. The space $gl(V )$
consisting of all the linear transformations on $V$ has a natural $\mathbb{Z}_2$-gradation as
\begin{equation}\label{lineartransformation}
    gl(V )_{\alpha} = \{f\in  gl(V )| f(V_{\alpha})\subseteq V_{\alpha+ \beta}, \alpha,\beta\in \mathbb{Z}_2 \}.
\end{equation}
\begin{df}
A bimodule over an alternative superalgebra $(A, \star)$ consists of a $\mathbb{Z}_2$-graded linear space $V$ and two even bilinear maps $\mathfrak{l},\mathfrak{r}:A \to gl(V)$
such that, for any homogeneous elements $x, y\in A$,
\begin{align}
\label{rephomalt1}&\mathfrak{l}(x\star y)+(-1)^{|x||y|}\mathfrak{l}(y\star x)-\mathfrak{l}(x)\mathfrak{l}(y)-(-1)^{|x||y|}\mathfrak{l}(y)\mathfrak{l}(x)=0,\\
\label{rephomalt2}&\mathfrak{r}(y)\mathfrak{r}(x)+(-1)^{|x||y|}\mathfrak{r}(x)\mathfrak{r}(y)-\mathfrak{r}(x\star y)-(-1)^{|x||y|}\mathfrak{r}(y\star x)=0,\\
\label{rephomalt3}&\mathfrak{r}(y)\mathfrak{r}(x)+(-1)^{|x||y|}\mathfrak{r}(y)\mathfrak{l}(x)-(-1)^{|x||y|}\mathfrak{l}(x)\mathfrak{r}(y)-\mathfrak{r}(x\star y)=0,\\
\label{rephomalt4}&\mathfrak{r}(y)\mathfrak{l}(x)+(-1)^{|x||y|}\mathfrak{l}(x\star y) - (-1)^{|x||y|}\mathfrak{l}(x)\mathfrak{l}(y)-\mathfrak{l}(x)\mathfrak{r}(y)=0.
\end{align}
\end{df}
\begin{prop}\label{semidirectproduct alt}
$(V,\mathfrak{l},\mathfrak{r})$ is a bimodule of an alternative superalgebra $(A, \star)$ if and only if the direct
sum $A\oplus V$ of vector spaces is turned into an alternative superalgebra (the semidirect sum) by defining
multiplication in $A \oplus V$ by
\begin{equation}
(x + a)\circ(y+ b) = x \star y + \mathfrak{l}(x)b +
\mathfrak{r}(y)a,\quad \forall x,y\in \mathcal{H}(A), a, b \in \mathcal{H}(V).
\end{equation}
We denote it by $A \ltimes_{\mathfrak{l},\mathfrak{r}} V$ or simply $A \ltimes V$.
\end{prop}
\begin{df}
Let $A$ be a Malcev superalgebra.
\begin{enumerate}
   \item [(i)] Let $V$ be a $\mathbb{Z}_2$-graded vector space. An even linear map  $\rho : A\to gl(V)$ (i.e., $|\r| = 0$) is a Malcev representation of $A$ on $V$ such that for any $x, y, z \in \mathcal{H}(A)$,
\begin{eqnarray}
\rho([[x, y], z])&=& \rho(x)\rho(y)\rho(z) - (-1)^{|z|(|x|+|y|)} \rho(z)\rho(x)\rho(y)\nonumber\\
&+&(-1)^{|x|(|y|+|z|)} \rho(y)\rho([z, x])\nonumber\\
&-& (-1)^{|x|(|y|+|z|)} \rho([y, z])\rho(x).
\end{eqnarray}
   \item [(ii)]Consider two Malcev representations  $\rho : A\to gl(V )$ and $\rho':A\to gl(V')$ of $A$, where $V$ and $V'$ are $\mathbb{Z}_2$-graded vector spaces.
We say that $\rho$  and $\rho'$ are equivalent if there exists a bijective even
linear map $\phi : V \to V'$ such that $\phi\circ \rho(x) = \rho'(x)\circ \phi, \forall x\in \mathcal{H}(A).$
\end{enumerate}
\label{df:superrepresentation}
\end{df}
\begin{prop}\label{semidirectmalcsuper}
$(V,\rho)$ is a representation of a Malcev superalgebra $A$ if and only if
 the direct sum $A\oplus V$ of vector spaces is a Malcev superalgebra by defining a map on $A\oplus V$ by
 \begin{equation}\label{sdp.2}
 [x+a,y+b]_{A\oplus V}:=[x,y]+\rho(x)b-(-1)^{|x||y|}\rho(y)a,
 \end{equation}
for all $x, y\in \mathcal{H}(A), a, b\in \mathcal{H}(V).$
In this case, $A\oplus V$ is called the \textit{semi-direct product} of $A$ and $V$, denoted by $A\ltimes_{\rho} V$ or simply $A\ltimes V$.
\end{prop}
Note that $(A\oplus V)_{\overline{0}}=A_{\overline{0}}\oplus V_{\overline{0}}$, implying that if $x + a \in \mathcal{H}(A\oplus V)$, then $|x + a| = |x| = |a|.$
\begin{proof}
For all $x,y,z,t\in \mathcal{H}(A)$ and $a,b,c,d\in \mathcal{H}(V)$, we have
\begin{align*}
 &[[x+a,z+c]_{A\oplus V},[y+b,t+d]_{A\oplus V}]_{A\oplus V}\\
 &=[[x,z],[y,t]]+\rho([x,z])\rho(y)d-(-1)^{|y||t|}\rho([x,z])\rho(t)b\\
 &\quad-(-1)^{(|x|+|z|)(|y|+|t|)}\rho([y,t])\rho(x)c +(-1)^{(|x|+|z|)(|y|+|t|)+|x||z|}\rho([y,t])\rho(z)a,
\\[0.2cm]
&[[[x+a,y+b]_{A\oplus V},z+c]_{A\oplus V},t+d]_{A\oplus V}\\
&=[[[x,y],z],t]+\rho([[x, y], z])d-(-1)^{|t|(|x|+|y|+|z|)}\rho(t)\rho([x,y])c\\
&\quad+(-1)^{|t|(|x|+|y|+|z|)+|z|(|x|+|y|)}\rho(t)\rho(z)\rho(x)b-(-1)^{|t|(|x|+|y|+|z|)+|z|(|x|+|y|)+|x||y|}\rho(t)\rho(z)\rho(y)a,
\\[0.2cm]
&[[[y+b,z+c]_{A\oplus V},t+d]_{A\oplus V},x+a]_{A\oplus V}\\
&=[[[y,z],t],x]+\rho([[y, z], t])a-(-1)^{|x|(|y|+|z|+|t|)}\rho(x)\rho([y,z])d\\
&\quad+(-1)^{|x|(|y|+|z|+|t|)+|t|(|y|+|z|)}\rho(x)\rho(t)\rho(y)c-(-1)^{|x|(|y|+|z|+|t|)+|t|(|y|+|z|)+|y||z|}\rho(x)\rho(t)\rho(z)b,
\\[0.2cm]
&[[[z+c,t+d]_{A\oplus V},x+a]_{A\oplus V},y+b]_{A\oplus V}\\
&=[[[z,t],x],y]+\rho([[z, t], x])b-(-1)^{|y|(|x|+|z|+|t|)}\rho(y)\rho([z,t])a\\
&\quad+(-1)^{|y|(|x|+|z|+|t|)+|x|(|z|+|t|)}\rho(y)\rho(x)\rho(z)d-(-1)^{|y|(|x|+|z|+|t|)+|x|(|z|+|t|)+|z||t|}\rho(y)\rho(x)\rho(t)c,
\\[0.2cm]
&[[[t+d,x+a]_{A\oplus V},y+b]_{A\oplus V},z+c]_{A\oplus V}\\
&=[[[t,x],y],z]+\rho([[t, x], y])c-(-1)^{|z|(|x|+|y|+|t|)}\rho(z)\rho([t,x])b\\
&\quad+(-1)^{|z|(|x|+|y|+|t|)+|y|(|x|+|t|)}\rho(z)\rho(y)\rho(t)a-(-1)^{|z|(|x|+|y|+|t|)+|y|(|x|+|t|)+|x||t|}\rho(z)\rho(y)\rho(x)d.
\end{align*}
Then, $(A\oplus V,[\c,\c]_{A\oplus V})$  is a Malcev superalgebra if and only if
\begin{eqnarray*}
\rho([[x, y], z])&=& \rho(x)\rho(y)\rho(z) - (-1)^{|z|(|x|+|y|)} \rho(z)\rho(x)\rho(y)\\
&+&(-1)^{|x|(|y|+|z|)} \rho(y)\rho([z, x])\\
&-& (-1)^{|x|(|y|+|z|)} \rho([y, z])\rho(x).
\qedhere\end{eqnarray*}
\end{proof}

\begin{prop}
Let $(V,\mathfrak{l},\mathfrak{r})$ be a bimodule of  an alternative superalgebra $(A,\star)$. Then,  $(V,\mathfrak{l}-(-1)^{|x||y|}\mathfrak{r})$ is a representation of the Malcev admissible  superalgebra $(A,[\c,\c])$ defined in Theorem \ref{thm:alttomalc}.
\end{prop}
\begin{proof}
By Proposition \ref{semidirectproduct alt}, $A \ltimes_{\mathfrak{l},\mathfrak{r}} V$ is an alternative superalgebra.
Consider  its associated Malcev  superalgebra $(A \oplus V, \overbrace{[\c, \c]})$,
\begin{align*}
 \overbrace{[x+a, y+b]}=&(x+a)\circ(y+b) -(-1)^{|x||y|}(y+b)\circ (x+a)\\
=& x\star y+ \mathfrak{l}(x)b+ \mathfrak{r}(y)a - (-1)^{|x||y|} y\star x -(-1)^{|a||y|}\mathfrak{l}(y)a -(-1)^{|x||b|} \mathfrak{r}(x)b \\
=&[x, y] + (\mathfrak{l}-(-1)^{|x||b|}\mathfrak{r})(x)b - (-1)^{|a||y|}(\mathfrak{l} -(-1)^{|a||y|}\mathfrak{r})(y)a.
\end{align*}
According to Proposition \ref{semidirectmalcsuper}, $(V,\mathfrak{l}-(-1)^{|x||y|}\mathfrak{r})$ is a representation of $(A,[\c, \c])$.
\end{proof}
\begin{ex}
Let $A = A_{\bar{0}}\oplus A_{\bar{1}}$ be a Malcev superalgebra. It is easy
to see that the map $ad : A\to gl(A)$ defined by $ad(x)(y ) = [x, y] ,
\forall x,y\in \mathcal{H}(A)$, is a Malcev representation of $A$ in itself. It is called the $\emph{adjoint representation}$ of $A$. Besides, this representation induces a Malcev
representation $ad : A_{\bar{0}}\to gl(A_{\bar{1}} )$ of the Malcev algebra $A_{\bar{0}}$ in the
vector space $A_{\bar{1}}$ defined by $ad(x)(y ) = [x, y] ,
\forall x\in A_{\bar{0}},y\in A_{\bar{1}}$ . It is called
the adjoint representation of $A_{\bar{0}}$ in $A_{\bar{1}}$.
\end{ex}
To articulate the adjoint representation and co-adjoint representation of a Malcev superalgebra, we
first present some basic facts about $\mathbb{Z}_{2}$-graded vector spaces.
Let $V=V_{\bar 0}\oplus V_{\bar 1}$ be a $\mathbb{Z}_{2}$-graded vector space over a field $\mathbb{K}$. Then $V^*=
Hom(V, \mathbb{K})$ is the dual vector space of $V$, whose $\mathbb{Z}_2$-gradation  is  given by
\begin{equation}\label{eq:dualgradation}
V^*_\alpha=\{a^*\in V^*\mid a^*(V_{\alpha+\bar1})=\{0\}\}, \;\;\forall \alpha\in \mathbb{Z}_2.
\end{equation}
Let $\langle\cdot ,\cdot \rangle:V^* \times V\rightarrow \mathbb K$ be the canonical pairing, which allows us to identify $V$ with $V^*$ by
\begin{equation}
\langle a^*,b\rangle=(-1)^{|a^*||b|}\langle b,a^*\rangle,\;\;\forall  a^*\in V^*,b\in V.
\end{equation}
We  extend $\langle\cdot,\cdot\rangle$ to $(V \otimes V)^* \times (V \otimes V)$~($(V \otimes V)^* =V^*\otimes V^*$ when $V$ is finite-dimensional) by setting
\begin{equation*}
\langle a_1^*\otimes a_2^*, b_1\otimes b_2\rangle=(-1)^{|a_2^*||b_1|}\langle a_1^*,  b_1\rangle\langle a_2^*,  b_2\rangle,\;\;\forall a_1^*, a_2^*\in V^*, b_1, b_2 \in V.
\end{equation*}

\begin{prop}
Let $(V,\rho)$ be a representation of a Malcev superalgebra $A$  and let $V^{*}$ denote the dual space of $V$.
Define a linear map $\rho^{*} : A \to gl(V^{*})$
by
\begin{equation}\label{dualrepresentation}
   \langle \r^*(x)a^*,b\rangle=-(-1)^{|x||a|}\langle a^*,\r(x)b\rangle,\;\forall x\in A, a^*\in
V^*, b\in V.
\end{equation}
Then $(V^{*},\rho^{*})$ is also a Malcev representation of $A$ in $V^{*}$, which is called $\emph{the dual representation}$
of $(V,\r )$.
\end{prop}
\begin{proof}
For any homogenous  $x,y,z\in A$, $a^{*}\in V^{*}$, $ b\in V$, we have
\begin{eqnarray*}
\langle \r^*([[x,y],z])a^*,b\rangle &=&-(-1)^{(|x|+|y|+|z|)|a|}\langle a^*,\r([[x,y],z])b\rangle,\\
\langle \rho^{*}(x)\rho^{*}(y)\rho^{*}(z)a^*,b\rangle&=&-(-1)^{(|x|+|y|+|z|)|a|+|x||y|+|x||z|+|y||z|}\langle a^*,\rho(z)\rho(y)\rho(x)b\rangle,\\
\langle\rho^{*}(z)\rho^{*}(x)\rho^{*}(y)a^*,b\rangle&=&-(-1)^{(|x|+|y|+|z|)|a|+|x||y|+|x||z|+|y||z|}\langle a^*,\rho(y)\rho(x)\rho(z)b\rangle,\\
 \langle\rho^{*}(y)\rho^{*}([z, x])a^*,b\rangle&=&(-1)^{(|x|+|y|+|z|)|a|+|x||y|+|y||z|}\langle a^*,\rho([z, x])\rho(y)b\rangle,\\
\langle\rho^{*}([y, z])\rho^{*}(x)a^*,b\rangle &=&(-1)^{(|x|+|y|+|z|)|a|+|x||y|+|x||z|}\langle a^*, \rho(x)\rho([y, z])b\rangle.
\end{eqnarray*}
Then we have for the map $\rho^{*}$,
\begin{align*}
\langle \r^*([[x,y],z])a^*,b\rangle &=\langle \big(\rho^{*}(x)\rho^{*}(y)\rho^{*}(z) - (-1)^{|z|(|x|+|y|)} \rho^{*}(z)\rho^{*}(x)\rho^{*}(y)\\
 &+ (-1)^{|x|(|y|+|z|)}\rho^{*}(y)\rho^{*}([z, x]) -(-1)^{|x|(|y|+|z|)} \rho^{*}([y, z])\rho^{*}(x)\big)a^*,b\rangle,
\end{align*}
what is equivalent to say that $\rho^{*}$ is a Malcev representation of $A$.
\end{proof}
\begin{ex}
Let $A = A_{\bar{0}}\oplus A_{\bar{1}}$ be a Malcev superalgebra and $A^{*}$ its
dual vector space. We easily show that the map $ad^{*} : A\to gl(A^{*})$
defined by
$$ \langle ad^*(x)\a,y\rangle=-(-1)^{|x||\a|}\langle \a,ad(x)y\rangle, \forall x,y\in \mathcal{H}(A), \a\in
A^*,$$
is a Malcev representation of $A$ in $A^{*}$. It is called $\emph{the co-adjoint representation}$ of $A$.
 Moreover, it is also clear that $ad^{*} : A_{\bar{0}}\to gl(A_{\bar{1}}^{*} )$ of the Malcev algebra $A_{\bar{0}}$ in the
vector space $A_{\bar{1}}^{*}$ defined by $\langle ad^*(x)\a,y\rangle=-\langle \a,ad(x)y\rangle ,
\forall x\in A_{\bar{0}},y\in A_{\bar{1}},\a\in A_{\bar{1}}^{*}$ . It is called
the co-adjoint representation of $A_{\bar{0}}$ in $A_{\bar{1}}^{*}$.
\end{ex}

\section{Super $\mathcal{O}$-operators of Malcev superalgebras and pre-Malcev superalgebras}%
In this section, we introduce the notion of super $\mathcal{O}$-operator and pre-Malcev superalgebras. Then we
study the relations among Malcev superalgebras, pre-Malcev superalgebras and
pre-alternative  superalgebras.

\begin{df}
A super $\mathcal{O}$-operator of alternative superalgebra $(A,\star)$ with respect to the bimodule $(V,\mathfrak{l},\mathfrak{r})$ is a linear map $T:V\to A$ such that,
for all $a, b \in \mathcal{H}(V)$,
\begin{equation}
\label{O-opalternative} T (a)\star T (b) = T \big(\mathfrak{l}(T (a))b + \mathfrak{r}(T (b))a\big).
 \end{equation}
 \end{df}
 A super Rota-Baxter operator of weight $0$ on an alternative superalgebra $(A,\star)$ is
just a super $\mathcal{O}$-operator associated to the bimodule $(A,L, R)$, where $L$ and $R$ are the left and right
multiplication operators corresponding to the multiplication $\star$.
\begin{df} \label{df:operator malcev}
Let $A$ be a Malcev superalgebra and $(V,\rho)$ a representation of $A$.  A linear map $T: V \rightarrow A$ with $|T|= 0$ is called a \textit{super $\mathcal {O}$-operator} of $A$ associated to $(V,\rho)$ if $T$ satisfies
\begin{equation}\label{eq:operator}
[T(a), T(b)] = T\big(\rho(T(a))b-(-1)^{|a||b|}\rho(T(b))a\big), \forall a, b\in \mathcal{H}(V).
\end{equation}

In particular, if $\mathcal {R}$ is a super $\mathcal {O}$-operator of $A$ associated to the adjoint representation $(A,ad)$, then $\mathcal{R}$ is called a \textit{super Rota-Baxter operator} (of weight 0) on $A$, that is, $\mathcal {R}$ satisfies
$$[\mathcal {R}(x), \mathcal {R}(y)] = \mathcal {R}\big([\mathcal {R}(x),y]+(-1)^{|x||y|}[x,\mathcal {R}(y)]\big),$$
for all $x, y \in \mathcal{H}(A)$.
\end{df}
\begin{prop}
Let $T$ be a super $\mathcal{O}$-operator of an alternative superalgebra $(A,\star)$ associated to a bimodule $(V,\mathfrak{l},\mathfrak{r})$.
Then,  $T$ is a super $\mathcal {O}$-operator of $(A,[\c,\c])$ with respect to $(V,\mathfrak{l}-(-1)^{|x||y|}\mathfrak{r})$.
\end{prop}
\begin{proof}
$T$ is a super $\mathcal{O}$-operator of $(A,[\c,\c])$ with respect to $(V,\mathfrak{l}-(-1)^{|x||y|}\mathfrak{r})$  since
\begin{align*}
[T(a),T(b)]=&T(a)\star T(b)-(-1)^{|a||b|}T(b)\star T(a)\\
=&T \big(\mathfrak{l}(T (a))b + \mathfrak{r}(T (b))a\big)-(-1)^{|a||b|}T \big(\mathfrak{l}(T (b))a + \mathfrak{r}(T (a))b\big)\\
=&T \big((\mathfrak{l}-(-1)^{|a||b|}\mathfrak{r})(T (a))b - (-1)^{|a||b|}(\mathfrak{l}-(-1)^{|a||b|}\mathfrak{r})(T (b))a\big).\qedhere
 \end{align*}
\end{proof}
\begin{df}
A pre-Malcev superalgebra $A$ is a super vector space
$A = A _{\bar{0}} \oplus A _{\bar{1}}$
equipped with a bilinear product $(x, y) \to x \c y$ satisfying
$$A_{\a}\c A_{\b}\subseteq A_{\a+ \b},\quad \a,\b\in \mathbb{Z}_{2},$$
and the following equations $(\forall x, y, z, t \in \mathcal{H}(A))$:
\begin{equation}
\begin{split}\label{PM}
 &(-1)^{|x|(|y|+|z|)}[y, z] \c (x \c t)+ [[x, y], z] \c t+ (-1)^{|x||y|} y \c ([x, z] \c t)\\
 &\qquad - x \c (y \c(z \c t)) +(-1)^{|z|(|x|+|y|)} z \c (x \c(y \c t))=0,
\end{split}
\end{equation}
  where $[x,y]=x\c y-(-1)^{|x||y|}y\c x.$
The identity \eqref{PM} is equivalent to
$PM(x, y, z, t) = 0$, where
\begin{equation}\label{PMexpanded}
\begin{split}
PM(x, y, z, t) &= (-1)^{|x|(|y|+|z|)} (y \cdot z) \cdot (x \cdot t) -(-1)^{|x|(|y|+|z|)+|y||z|} (z \cdot y) \cdot (x \cdot t)\\
&+ ((x \cdot y) \cdot z) \cdot t - (-1)^{|x||y|}((y \cdot x) \cdot z) \cdot t\\
 &-(-1)^{(|x|+|y|)|z|} (z \cdot (x \cdot y)) \cdot t + (-1)^{|x||y|+(|x|+|y|)|z|} (z \cdot (y \cdot x)) \cdot t\\
& +(-1)^{|x||y|} y \cdot ((x \cdot z) \cdot t) -(-1)^{|x|(|y|+|z|)} y\cdot ((z \cdot x) \cdot t)\\
 &- x \cdot (y  \cdot (z \cdot t)) + (-1)^{|z|(|x|+|y|)}z \cdot (x \cdot (y \cdot t)),
\end{split}
\end{equation}
for all $x,y,z,t\in \mathcal{H}(A).$
\label{df:pre malcev superalg}
 \end{df}
\begin{thm}
Let $(A,\c)$ be a pre-Malcev superalgebra.
\begin{enumerate}
\item[(i)] \label{item1:thm:supercommutator} Let us define, for any homogeneous elements $x, y \in A$ the operation
\begin{equation}\label{eq:supercommutator}
[x, y] = x \c y- (-1)^{|x||y|}y\c x.
\end{equation}
Then $(A,[\c,\c])$ is a Malcev superalgebra, which is called the sub-adjacent Malcev
superalgebra of $(A,\c)$ and denoted by $A^{C}$ . We call $(A,\c)$ a compatible
pre-Malcev superalgebra structure on the Malcev superalgebra.
\item[(ii)] \label{item2:thm:supercommutator} Let $L_{x}$ (for any $x\in \mathcal{H}(A)$) denote the left multiplication
operator, i.e., $L_{x}(y)=x\c y, \forall y\in \mathcal{H}(A)$. Then $ L \colon A \rightarrow A$ with $x \rightarrow L_{x}$
gives a representation of the Malcev superalgebra $A$, that is,
\[
                    L_{[[x,y],z]} = L_{x}L_{y}L_{z} - (-1)^{|z|(|x|+|y|)}L_{z}L_{x}L_{y} + (-1)^{|x|(|y|+|z|)}L_{y}L_{[z,x]} - (-1)^{|x|(|y|+|z|)}L_{[y,z]}L_{x},
                \]
 for all $x,y,z \in \mathcal{H}(A)$.
\end{enumerate}
\label{thm:supercommutator}
\end{thm}
\begin{proof}
For part (i) we have, for any $x,y,z,t\in \mathcal{H}(A)$,
\begin{align*}
&(-1)^{|y||z|}[[x, z], [y, t]] - [[[x, y], z], t] -(-1)^{|x|(|y|+|z|+|t|)} [[[y, z], t], x]\\
&\qquad- (-1)^{(|x|+|y|)(|z|+|t|)} [[[z, t], x], y]- (-1)^{|t|(|x|+|y|+|z| )} [[[t, x], y], z]\\
&=(-1)^{|y||z|}[x,z]\c(y\c t)-(-1)^{|y|(|z|+|t|)}[x,z]\c(t\c y)-(-1)^{|y||x|+|t|(|x|+|z|)} [y, t]\c(x\c z)\\
&\qquad +(-1)^{|x|(|y|+|z|+|t|)+|t||z|} [y, t]\c(z\c x)-[[x, y], z]\c t+ (-1)^{|t|(|x|+|y|+|z|)} t\c([x, y]\c z)\\
&\qquad -(-1)^{|t|(|x|+|y|+|z|)+(|x|+|y|)|z|} t\c(z\c(x\c y))+ (-1)^{|t|(|x|+|y|+|z|)+ |x||y|+(|x|+|y|)|z|} t\c(z\c(y\c x))\\
&\qquad -(-1)^{|x|(|y|+|z|+|t|)}[[y, z], t]\c x+ x\c([y, z]\c t)- (-1)^{|t|(|y|+|z|)} x\c(t\c(y\c z))\\
&\qquad + (-1)^{|t|(|y|+|z|) +|y||z|} x\c(t\c(z\c y))- (-1)^{(|x|+|y|)(|z|+|t|)} [[z, t], x]\c y\\
&\qquad + (-1)^{|x|(|y|+|z|+|t|)} y\c([z, t]\c x) - (-1)^{|x||y|} y\c(x\c(z\c t))\\
&\qquad + (-1)^{|x|(|y|+|z|+|t|)+|z||t|} y\c(x\c(t\c z))\\
&\qquad - (-1)^{|t|(|x|+|y|+|z| )} [[t, x], y]\c  z+(-1)^{(|x|+|y|)(|z|+|t|)} z\c([t, x]\c y)\\
&\qquad - (-1)^{|x|(|y|+|z|+|t|)+ |y||z|}z\c(y\c(t\c x))+ (-1)^{|x||y|+(|x|+|y|)|z|}z\c(y\c(x\c t))\\
&= HPM(x, t, y, z)  + HPM(y, x, z, t) + HPM(z, y, t, x)  + HPM(t, z, x, y) = 0.
 \end{align*}
For (ii) we compute
    \begin{align*}
 & \Big( L_{[[x,y],z]} - L_{x}L_{y}L_{z} + (-1)^{|z|(|x|+|y|)}L_{z}L_{x}L_{y} - (-1)^{|x|(|y|+|z|)}L_{y}L_{[z,x]}\\
&\qquad + (-1)^{|x|(|y|+|z|)}L_{[y,z]}L_{x} \Big) (t)\\
 &=[[x,y],z] \cdot t - x \cdot(y \cdot (z \cdot t)) + (-1)^{|z|(|x|+|y|)}z \cdot (x \cdot (y \cdot t))\\
 &\qquad- (-1)^{|x|(|y|+|z|)} y \cdot ([z,x] \cdot t) + (-1)^{|x|(|y|+|z|)}[y,z] \cdot (x \cdot t)\\
 &=  PM(x,y,z,t) = 0.\qedhere
    \end{align*}
\end{proof}
\begin{prop}
    A superalgebra $(A,\c)$  is a pre-Malcev superalgebra if and only if
    $(A,[\c,\c])$ with the commutator $[x,y] = x \cdot y - (-1)^{|x||y|} y \cdot x$ is a Malcev superalgebra
    and $L$ is a representation of $(A,[\c,\c])$.
\end{prop}
\begin{proof}
It follows from Theorem~\ref{thm:supercommutator} above and the definitions of Malcev superalgebra and representation of a Malcev superalgebra.
\end{proof}

Furthermore, we can construct pre-Malcev superalgebras from super $\mathcal{O}$-operators
of Malcev superalgebras.
\begin{thm}\label{tmm:malcevsuperalgtopremalcevsuperalg}
  Let $(A, [\c,\c])$ be a Malcev superalgebra and $(V, \rho)$ be its representation.
Let  $T: V\to A$  be a super $\mathcal{O}$-operator associated to
$(V, \rho)$. Then there exists a
pre-Malcev superalgebra structure on $V$ given by
\begin{eqnarray}\label{malcevsuperalg==>Premalcevsuperalg}
a\cdot b=\rho(T(a))b,\qquad \forall a,b\in A.
\end{eqnarray}
\end{thm}
\begin{proof}
Let $a,b,c,d \in  V$. Set
$x = T(a), y = T(b); z = T(c)$
and
$[a,b] = a \cdot b - (-1)^{|a||b|} b \cdot a.$ Hence, we have
\begin{align*}
&(-1)^{|a|(|b|+|c|)}[b, c] \cdot (a \cdot d)+ [[a, b], c] \cdot d+ (-1)^{|a||b|} b \cdot ([a, c] \cdot d)\\
&\quad - a \cdot (b \cdot(c \cdot d)) +(-1)^{|c|(|a|+|b|)} c \cdot (a \cdot(b \cdot d))\\
&=(-1)^{|a|(|b|+|c|)}\big(\rho(T(b))c-(-1)^{|b||c|}\rho(T(c))b\big)\cdot( \rho(T(a))d)+[\rho(T(a))b- (-1)^{|a||b|}\rho(T(b))a,c]\cdot d\\
&\quad+ (-1)^{|a||b|} b \cdot ( (\rho(T(a))c- (-1)^{|a||c|}\rho(T(c))a)\cdot d)-a \cdot (b \cdot \rho(T(c))d)\\
&\quad +(-1)^{|c|(|a|+|b|)} c \cdot (a \cdot \rho(T(b))d)\\
&=(-1)^{|a|(|b|+|c|)}\rho\big(T(\rho(T(b))c-(-1)^{|b||c|}\rho(T(c)))b\big)\rho(T(a))d+\Big(\rho(T(\rho(T(a))b\\
&\quad- (-1)^{|a||b|}\rho(T(b))a))c-(-1)^{|c|(|a|+|b|)}\rho(T(c))(\rho(T(a))b\\
&\quad- (-1)^{|a||b|}\rho(T(b))a)\Big)\cdot d+ (-1)^{|a||b|} b \cdot \big(\rho(T( \rho(T(a))c- (-1)^{|a||c|}\rho(T(c))a))d\big)\\
&\quad-a \cdot (\rho(T(b))\rho(T(c))d)+(-1)^{|c|(|a|+|b|)} c \cdot (\rho(T(a))\rho(T(b))d)\\
&=(-1)^{|a|(|b|+|c|)}\rho([T(b), T(c)])\rho(T(a))d+\rho\Big(T\Big(\rho([T(a),T(b)])c-(-1)^{|c|(|a|+|b|)}\rho(T(c))(\rho(T(a))b\\
&\quad- (-1)^{|a||b|}\rho(T(b))a)\Big)\Big)d+ (-1)^{|a||b|} \rho(T(b))\rho([T(a),T(c)])d-\rho(T(a))\rho(T(b))\rho(T(c))d\\
&\quad+(-1)^{|c|(|a|+|b|)} \rho(T(c))\rho(T(a))\rho(T(b))d\\
&=(-1)^{|a|(|b|+|c|)}\rho([T(b), T(c)])\rho(T(a))d+\rho([[T(a),T(b)],T(c)])d+ (-1)^{|a||b|} \rho(T(b))\rho([T(a),T(c)])d\\
&\quad-\rho(T(a))\rho(T(b))\rho(T(c))d+(-1)^{|c|(|a|+|b|)} \rho(T(c))\rho(T(a))\rho(T(b))d\\
&=(-1)^{|x|(|y|+|z|)}\rho([y, z])\rho(x)d+\rho([[x,y],z])d+ (-1)^{|x||y|} \rho(y)\rho([x,z])d\\
&\quad-\rho(x)\rho(y)\rho(z)d+(-1)^{|z|(|x|+|y|)} \rho(z)\rho(x)\rho(y)d=0.
\end{align*}
Thus, $(V,\cdot)$ is a pre-Malcev superalgebra.
\end{proof}
Therefore, there exists a Malcev superaglebra structure on $V$ given by \eqref{eq:supercommutator}
and $T$ is a homomorphism of Malcev superalgebras. Furthermore, there is an
induced pre-Malcev superalgebra structure on $T(V )$ given by
\begin{equation}\label{eq:homomorphism}
T(a) \c T(b) = T(a \cdot b), \forall a,b \in V.
\end{equation}
Moreover, the corresponding associated Malcev superalgebra structure on $T(V )$
given by \eqref{eq:supercommutator} is just a Malcev supersubalgebra structure of $A$ and $T$ becomes a
homomorphism of pre-Malcev superalgebra.
\begin{prop}\label{pro:nsc}
   Let $(A, [\c,\c])$ be a Malcev superalgebra. Then there exists  a compatible pre-Malcev superalgebra structure on $A$ if and only if there exists an invertible super $\mathcal{O}$-operator on $A$ associated to  a representation $(V, \rho)$.
\end{prop}
\begin{proof}
Let $(A,\cdot)$ be a pre-Malcev superalgebra and $(A,[\c,\c])$ be the associated Malcev superalgebra.  Then the identity map $id: A \to A$ is an invertible super $\mathcal{O}$-operator  on  $(A,[\c,\c])$  associated to $(A,ad)$.

Conversely, suppose that there exists an invertible super $\mathcal{O}$-operator $T$ of $(A,[\c,\c])$  associated to  a representation $(V, \r)$ . Then, using Theorem \ref{tmm:malcevsuperalgtopremalcevsuperalg},  there is a pre-Malcev superalgebra structure on $T(V)=A$ given by
\begin{equation*}
   T(a)\cdot T(b)=T(\rho(T(a))b),\ \ \forall\ a, b\in \mathcal{H}(V).
\end{equation*}
If we set $x=T(a)$ and $y=T(b)$, then we get
\begin{equation*}
  x\cdot y=T(\rho(x)T^{-1}(y)),\qquad \forall x,y\in \mathcal{H}(A).
\end{equation*}
\end{proof}
An obvious consequence of Theorem \ref{tmm:malcevsuperalgtopremalcevsuperalg} is the following construction of a pre-Malcev  superalgebra in terms of super Rota-Baxter operator on a Malcev superalgebra.
\begin{cor}\label{preMalcevByRotaBaxter}
  Let $(A, [\c,\c])$ be a Malcev superalgebra and the linear map $\mathcal{R}: A\rightarrow A$  is  a  super Rota-Baxter operator. Then there exists a pre-Malcev superalgebra structure on $A$ given by
\begin{eqnarray*}
 x\cdot y=[\mathcal{R}(x),y],\ \forall \ x,y\in \mathcal{H}(A).
\end{eqnarray*}
If in addition, $\mathcal{R}$ is invertible, then there is a compatible pre-Malcev superalgebra structure on $A$ given by
\begin{equation*}
   x\cdot y=\mathcal{R}([x,\mathcal{R}^{-1}(y)]),\ \forall x,y \in \mathcal{H}(A).
\end{equation*}
\end{cor}
Let us recall some basic notions about bilinear forms of a Malcev superalgebra  $(A, [\c,\c])$.
  A bilinear form $\omega:A\otimes A\to \mathbb{K}$  is said to be
  \begin{enumerate}
   \item supersymmetric if  $\omega(x, y) = (-1)^{|x||y|}\omega(y, x), \forall x,y\in \mathcal{H}(A),$
   \item skew supersymmetric if  $\omega(x, y) = -(-1)^{|x||y|}\omega(y, x), \forall x,y\in \mathcal{H}(A),$
   \item nondegenerate if $x\in A$ satisfies $\omega(x,y)=0, \forall y\in \mathcal{H}(A)$, then $x = 0,$
   \item invariant if  \begin{equation}\omega([x, y], z) = \omega(x, [y, z]),\forall x, y, z \in \mathcal{H}(A).\end{equation}
    \end{enumerate}
\begin{df}
 Let $(A,[\c,\c])$ be  a Malcev superalgebra. A  skew supersymmetric, nondegenerate
 bilinear form $\omega:A\otimes A\to \mathbb{K}$   is called a super two-cocycle
on $A$ if $\omega$ satisfies
\begin{equation}\label{eq:2-cocycle}
(-1)^{|x||z|}\omega(x, [ y, z]) + (-1)^{|y||x|}\omega(y, [ z, x])+(-1)^{|z||y|}\omega(z, [ x, y]),
\end{equation}
$\forall x,y,z\in \mathcal{H}(A)$. Then $\omega$ is called symplectic. A Malcev superalgebra
$(A,[\c,\c])$ with a symplectic form is called a symplectic Malcev superalgebra and denoted $(A,[\c,\c],\omega)$.
\end{df}
Let $(A,[\c,\c],\omega)$ be a symplectic Malcev superalgebra. Define the multiplication
$"\c"$ on $A$ by
\begin{equation}\label{eq:symplectictopremalc}
\omega(x\cdot y,z)=(-1)^{|x|(|y|+|z|)}\omega(y,[z,x] ).
\end{equation}
\begin{thm}
Under the above notations,
 there exists a compatible pre-Malcev superalgebra structure $"\cdot"$ on $A$ given by Eq. \eqref{eq:symplectictopremalc}.
\end{thm}

 \begin{proof}
 Define the linear map $T : A\to A^{*}$ by $\langle T(x), y\rangle = \omega(x, y)$.    Since $(A,[\c,\c] )$ is a Malcev superalgebra, $(A^*,ad^\ast)$ is a representation of $A$. By the fact that $\omega$ is skew supersymmetric and using  \eqref{eq:2-cocycle}, we obtain that $T$ is an invertible  super $\mathcal{O}$-operator  associated to the representation  $(A^{*},ad^\ast)$. By Proposition~\ref{pro:nsc}, there exists a compatible pre-Malcev superalgebra structure given by
\begin{equation*}
x\cdot y = T^{-1}(ad^{\ast}(x)T(y)).
\end{equation*}
Hence,
\begin{align*}
\omega(x\cdot y,z)=&\langle T(x\cdot y),z\rangle=\langle ad^{\ast}(x)T(y),z\rangle=-(-1)^{|x||y|}\langle T(y),[x, z]\rangle\\
&=-(-1)^{|x||y|}\omega(y, [x, z])=(-1)^{|x|(|y|+|z|)}\omega(y,[z, x]).
\qedhere
\end{align*}
\end{proof}

Analogous to the connection between alternative superalgebras and Malcev
superalgebras, pre-alternative superalgebras are closely related to pre-Malcev superalgebras.
\begin{df}
A pre-alternative superalgebra $A$ is a super vector space $A = A _{\bar{0}} \oplus A _{\bar{1}}$
equipped with two bilinear maps $\prec,\succ: A \otimes A \rightarrow A$,   satisfying
for all $x, y, z \in \mathcal{H}(A)$ and  $x \star y = x \prec y + x\succ y$,
\begin{eqnarray}
&&(x\star y) \succ z - x\succ(y \succ z) + (-1)^{|x||y|}(y \star x) \succ z - (-1)^{|x||y|} y \succ (x \succ z) = 0,\\
&&(x\prec y)\prec z - x\prec(y \star z) + (-1)^{|y||z|}(x \prec z)\prec y -(-1)^{|y||z|} x\prec(z\star y) = 0,\\
&&(x\succ y) \prec z - x\succ(y \prec z) + (-1)^{|x||y|}(y \prec x) \prec z - (-1)^{|x||y|}y \prec (x \star z) = 0,\\
&&(x\succ y)\prec z - x\succ(y \prec z) + (-1)^{|y||z|}(x \star z)\succ y -(-1)^{|y||z|} x\succ(z\succ y) = 0.
\end{eqnarray}
\end{df}
\begin{prop}
Let $(A,\prec,\succ)$ be a pre-alternative superalgebra. Then the product  $x \star y = x \prec y + x\succ y$ defines an alternative superalgebra $A$.
\end{prop}
\begin{thm}\label{pre-altToPre-Malcev}
    Let $T:V\to A$ be a super  $\mathcal{O}$-operator of alternative superalgebra $(A,\star)$ with respect to the bimodule $(V,\mathfrak{l},\mathfrak{r})$. Then $(V,\prec,\succ)$ is a pre-alternative superalgebra, where for all $a,b\in \mathcal{H}(V)$,
\begin{equation}
  \label{alt==>prealt} a\succ b= \mathfrak{l}(T (a))b \ \ \text{and}\ \ a\prec b= \mathfrak{r}(T (b))a.
\end{equation}
 Moreover, if $(V,\cdot)$ is the   pre-Malcev  superalgebra associated to the  Malcev admissible superalgebra $(A,[\c,\c])$ on the representation $(V,\mathfrak{l}-(-1)^{|x||y|}\mathfrak{r})$, then $a\cdot b=a\succ b-(-1)^{|a||b|}b\prec a$.

\end{thm}

\begin{proof} For any $a,b,c\in \mathcal{H}(V)$, using \eqref{rephomalt3} and \eqref{O-opalternative} yields
\begin{align*}
&(a\succ b) \prec c -a\succ(b \prec c) + (b \prec a) \prec c-b \prec (a \star c)\\
&\quad =\mathfrak{r}(T (c))\mathfrak{l}(T (a))b-\mathfrak{l}(T(a))\mathfrak{r}(T(c))b \\
&\quad\quad
+(-1)^{|a||b|}\mathfrak{r}(T(c))\mathfrak{r}(T (a))b-(-1)^{|a||b|}\mathfrak{r}(T(a\succ c+ a\prec c))b\\
&\quad = \mathfrak{r}(T (c))\mathfrak{l}(T (a))b-\mathfrak{l}(T(a))\mathfrak{r}(T(c))b \\
&\quad\quad
+(-1)^{|a||b|}\mathfrak{r}(T(c))\mathfrak{r}(T (a))b-(-1)^{|a||b|}\mathfrak{r}(T(\mathfrak{l}(T(a))c+ \mathfrak{r}(T(c))a))b\\
&\quad =\mathfrak{r}(T (c))\mathfrak{l}(T (a))b-\mathfrak{l}(T(a))\mathfrak{r}(T(c))b \\
&\quad\quad
+(-1)^{|a||b|}\mathfrak{r}(T(c))\mathfrak{r}(T (a))b-(-1)^{|a||b|}\mathfrak{r}(T(a)\star T(c))=0.
\end{align*}
The other identities for $(V,\prec,\succ)$ being a pre-alternative superalgebras can be verified
similarly.
Moreover, using Eqs.  \eqref{malcevsuperalg==>Premalcevsuperalg} and  \eqref{alt==>prealt},
\begin{equation*}
a\cdot b=(\mathfrak{l}-(-1)^{|a||b|}\mathfrak{r})(T(a))b=\mathfrak{l}(T(a))b-(-1)^{|a||b|}\mathfrak{r}(T(a))b=a\succ b-(-1)^{|a||b|}b\prec a.\qedhere
 \end{equation*}
 \end{proof}
\begin{cor}\label{homalt==>homprealt}
Let $(A, \star )$ be an alternative superalgebra and $\mathcal{R}:A\rightarrow A$ be a super Rota-Baxter operator
of weight $0$. If multiplications
$\prec $ and $\succ $ on $A$ are defined for all $x, y\in \mathcal{H}(A)$ by
$x\prec y = x\star \mathcal{R}(y)$ and $x\succ y = \mathcal{R}(x)\star y$,
then $(A, \prec , \succ)$ is a pre-alternative superalgebra.

Moreover, if $(A,\cdot)$ is the   pre-Malcev  superalgebra associated to the  Malcev admissible superalgebra $(A,[\c,\c])$, then $x\cdot y=x\succ y-(-1)^{|x||y|}y\prec x$.
\end{cor}
Summarizing the above study, we have the following commutative diagram
of categories:

\begin{equation}\label{Diagramme1}
\begin{split}
\xymatrix{
\begin{array}{c}
\substack{\text{Pre-alternative} \\ \text{superalg.}}
\end{array} \ar[rr]_{\quad x\prec y + x\succ y} \ar[dd]^{x\succ y - (-1)^{|x||y|}y\prec x} && \begin{array}{c}
\substack{\text{Alternative} \\ \text{superalg.}}
\end{array}
\ar@<-1ex>[ll]_{\quad super R-B }\ar[dd]^{x\star y - (-1)^{|x||y|}y\star x}\\
&& \\
\begin{array}{c}
\substack{\text{Pre-Malcev} \\ \text{superalg.}}
\end{array}
\ar[rr]_{\quad x\c y - (-1)^{|x||y|}y\c x} &&
\begin{array}{c}
\substack{\text{Malcev} \\ \text{superalg.}}
\end{array}\ar@<-1ex>[ll]_{\quad super R-B }
}
\end{split}
\end{equation}

\section{An analogue of the super CYBE on Malcev superalgebras  }
We derive a close relation between super $\mathcal{O}$-operators and solutions to the super Malcev  Yang-Baxter equations on Malcev superalgebras. We prove that under
the skew-supersymmetric  condition, a solution to the super MYBE in a Malcev superalgebra
is equivalent to a super $\mathcal{O}$-operator associated to the co-adjoint  representation. Further, we give
an explicit constructing way to exploit  solutions to the super MYBE in semidirect
product Malcev superalgebras from a common given super $\mathcal{O}$-operator.
\subsection{Super $\mathcal{O}$-operators and operator forms of the super MYBE}
Let $(A, [\c,\c])$ be a Malcev superalgebra and $r\in A \otimes A$. The standard form of the super
Malcev Yang-Baxter equation  in $A$ is
given in the tensor expression as follows:
\begin{equation}\label{eq:super MYBE}
[r_{12},r_{13}]+[r_{12},r_{23}]+[r_{13},r_{23}]= 0,
\end{equation}
where $r$ is called a solution of the super MYBE
and for $r=\sum\limits_i x_i\otimes y_i\in A^{\otimes 2},$
\begin{equation}
r_{12}=\sum_ix_i\otimes y_i\otimes 1,\quad r_{13}=\sum_{i}x_i\otimes
1\otimes y_i,\quad r_{23}=\sum_i1\otimes x_i\otimes y_i.
\label{eq:r12}
\end{equation}
Note that Eq. \eqref{eq:super MYBE} has the same form as the ordinary case, but the commutation
relation is in the graded spaces, that is,
\begin{align}
[r_{12},r_{13}]&=\sum_{i,j}(-1)^{|x_j||y_i|}[x_i,x_j]\otimes y_i\otimes
y_j,\\
 [r_{13}, r_{23}]&=\sum_{i,j}x_i\otimes x_j\otimes
[y_i, y_j],\\
[r_{12}, r_{23}]&=\sum_{i,j} (-1)^{|x_j||y_i|}x_i\otimes
[y_i,  x_j]\otimes y_j.\end{align}
Let $V$ be a super vector space. Then $r=\sum\limits_i x_i\otimes y_i\in
A^{\otimes 2}$ is said to be {\bf skew-supersymmetric} (resp. {\bf supersymmetric}) if
$r=-\sigma(r)$ (resp. $r=\sigma(r)$).
Furthermore, $r$ can be regarded as a map from $V^{*}$ to $V$ in the following way
\begin{align}
\langle x^*\otimes y^*,r\rangle&=\langle x^*\otimes y^*, \sum\limits_i x_i\otimes y_i\rangle=\sum_{i}(-1)^{|y^*||x_i|}\langle x^*,x_i\rangle\langle  y^*,y_i\rangle\nonumber\\
&=(-1)^{|r||x^*|}\langle y^*,r(x^*)\rangle,\quad
\forall x^*,y^*\in A^*.\label{eq:idenrmap}
\end{align}
Eq. \eqref{eq:super MYBE} gives the tensor form of  super MYBE. What we will do next is
to replace the tensor form by a linear operator satisfying some conditions.
\begin{thm}\label{mybe==dualoop}
Let $(A, [\c,\c])$ be a Malcev superalgebra and $r\in A \otimes A$ be  skew-supersymmetric with $|r|=0$ . Then $r$ is
a  solution of  the super MYBE in $A$ if and only if $r$ satisfies
\begin{equation}
[r(x^{*}), r(y^{*})] = r\big(ad^{*}r(x^{*})(y^{*}) -(-1)^{|x^{*}||y^{*}|} ad^{*}r(y^{*})(x^{*})\big), \forall x^{*}, y^{*}\in A^{*}.
\label{rOp}
\end{equation}
So there is a  pre-Malcev superalgebra structure on $A^{*}$ given by
\begin{equation}
 x\cdot y=ad^{*}r(x^{*})(y^{*}),\ \forall \ x^{*},y^{*}\in A^{*}.
\end{equation}
\end{thm}

\begin{proof}
 Let $\{e_1,...,e_m,f_1,...,f_n\}$ be a basis of $A$ where $\{e_1, . . . , e_m\}$ is a basis
of $A_{\bar{0}}$ and $\{f_1, . . . , f_n\}$ is a basis of $A_{\bar{1}}$ and let
$\{e_1^*,...,e_m^*,f_1^*,...,f_n^*\}$ be its dual basis with $e^{*} _{i}\in A^{*} _{\bar{0}}$
and $f^{*} _{j}
\in A^{*} _{\bar{1}}$. Since $r$ is skew-supersymmetric
and $|r| = 0$, we can set
$$r=\sum_{i,j}^{m}{a_{ij}e_i\otimes
e_j}+ \sum_{k,l}^{n}{b_{kl}f_k\otimes
f_l},$$
where $a_{ij}=-a_{ji}$ and $b_{kl} = b_{lk}$.
Now, we have
\begin{align*}
[r_{12},r_{13}]&=\left[\sum_{i,j}^{m}a_{ij}e_i\otimes
e_j\otimes 1+\sum_{k,l}^{n} b_{kl}f_k\otimes
f_l \otimes 1 ,\sum_{p,q}^{m}a_{pq}e_p\otimes 1\otimes
e_q+\sum_{s,t}^{n} b_{st}f_s\otimes 1\otimes
f_t\right]\\
&=\sum_{i,j,k,l,p}^{m} C_{ip}^{a} a_{ij}a_{pq} e_a\otimes e_j\otimes e_q + \sum_{i,j}^{m}\sum_{s,t,b}^{n}C_{is}^{b} a_{ij}b_{st} f_b\otimes e_j\otimes f_t\\
&+ \sum_{p,q}^{m}\sum_{k,l,c}^{n} C_{kp}^{c} b_{kl}a_{pq} f_c\otimes f_l\otimes e_q-\sum_{d}^{m}\sum_{k,l,s,t}^{n}C_{ks}^{d} b_{kl}b_{st} e_d\otimes f_l\otimes f_t,
\end{align*}
\begin{align*}
[r_{12},r_{23}]&=\left[\sum_{i,j}^{m}a_{ij}e_i\otimes
e_j\otimes 1+\sum_{k,l}^{n} b_{kl}f_k\otimes
f_l \otimes 1 ,  \sum_{p,q}^{m} 1\otimes a_{pq}e_p\otimes
e_q+\sum_{s,t}^{n} 1\otimes b_{st}f_s\otimes
f_t\right]\\
&=\sum_{i,j,k,l,p}^{m} C_{jp}^{a} a_{ij}a_{pq} e_i\otimes e_a\otimes e_q + \sum_{i,j}^{m}\sum_{s,t,b}^{n}C_{js}^{b} a_{ij}b_{st} e_i\otimes f_b\otimes f_t\\
&+ \sum_{p,q}^{m}\sum_{k,l,c}^{n} C_{lp}^{c} b_{kl}a_{pq} f_k\otimes f_c\otimes e_q+\sum_{d}^{m}\sum_{k,l,s,t}^{n}C_{ps}^{d} b_{kl}b_{st} f_k\otimes e_d\otimes f_t,
\end{align*}
\begin{align*}
[r_{13},r_{23}]&=\left[\sum_{i,j}^{m}a_{ij}e_i\otimes 1\otimes
e_j+\sum_{k,l}^{n} b_{kl}f_k\otimes 1\otimes
f_l ,\sum_{p,q}^{m} 1\otimes a_{pq}e_p\otimes
e_q+\sum_{s,t}^{n} 1\otimes b_{st}f_s\otimes
f_t\right]\\
&=\sum_{i,j,k,l,p}^{m} C_{jq}^{a} a_{ij}a_{pq} e_i\otimes e_p\otimes e_a + \sum_{i,j}^{m}\sum_{s,t,b}^{n}C_{jt}^{b} a_{ij}b_{st} e_i\otimes f_s\otimes f_b\\
&+ \sum_{p,q}^{m}\sum_{k,l,c}^{n} C_{lq}^{c} b_{kl}a_{pq} f_k\otimes e_p\otimes f_c-\sum_{d}^{m}\sum_{k,l,s,t}^{n}C_{lt}^{d} b_{kl}b_{st} f_k\otimes f_s\otimes e_d,
\end{align*}
where $C ^{k}
_{ij}$'s are the structure coefficients of Malcev superalgebra $A$ on the basis
$\{e_1,...,e_m,f_1,...,f_n\}$.
Then $r$ is a solution of the super
MYBE in $(A,[\c,\c])$ if and only if (for
any $1\leq a,d\leq m$ and $1\leq b,c\leq n$)
\begin{align*}
&\sum_{i,p}\Big( C_{ip}^{a} a_{ij}a_{pq}+C_{pi}^{q} a_{ip}a_{ji} + C_{ip}^{a} a_{ai}a_{pq}\Big)e_a\otimes e_j\otimes e_q\\
&\quad +\Big(\sum_{i}^{m}\sum_{s}^{n}C_{is}^{b} a_{ij}b_{st}+\sum_{l,s}^{n}C_{ls}^{j} b_{bl}b_{st}+ \sum_{i}^{m}\sum_{s}^{n}C_{is}^{t} a_{ij}b_{bs}\Big) f_b\otimes e_j\otimes f_t\\
&\quad+\Big(\sum_{p}^{m}\sum_{k}^{n} C_{kp}^{c} a_{pq}b_{kl} + \sum_{p}^{m}\sum_{k}^{n} C_{kp}^{l} b_{ck}a_{pq}-\sum_{k,t}^{n}C_{kt}^{q} b_{ck}b_{lt} \Big)f_c\otimes f_l\otimes e_q\\
&\quad+\Big(-\sum_{d}^{m}\sum_{k,s}^{n}C_{ks}^{d} b_{kl}b_{st} + \sum_{j}^{m}\sum_{s}^{n}C_{js}^{l} a_{dj}b_{st}+ \sum_{j}^{m}\sum_{s}^{n}C_{js}^{t} a_{dj}b_{ls}\Big) e_d\otimes f_l\otimes f_t =0.
\end{align*}
On the other hand, by \eqref{eq:idenrmap}, we get
$$r(e_j^*)=-\sum\limits_{i=1}^{m} a_{ji}e_i,\quad r(f_k^*)=-\sum\limits_{s=1}^{n} b_{ks}f_s,\quad 1\leq j\leq m,  1\leq k\leq n .$$
We prove the conclusion in the following four cases:

$\textbf{Case (1)}$ $x^{*}=e_{j}^{*}$ and $y^{*}=e_{q}^{*}$.Then by Eq.\eqref{rOp}, we have:
$$\sum_{i,p}\Big( C_{ip}^{a} a_{ij}a_{pq}+C_{pi}^{q} a_{ip}a_{ji} + C_{ip}^{a} a_{ai}a_{pq}\Big)e_a=0.$$

$\textbf{Case (2)}$ $x^{*}=e_{q}^{*}$ and $y^{*}=f_{c}^{*}$.
Then by Eq.\eqref{rOp}, we have:
$$\Big(\sum_{p}^{m}\sum_{k}^{n} C_{kp}^{c} a_{pq}b_{kl} + \sum_{p}^{m}\sum_{k}^{n} C_{kp}^{l} b_{ck}a_{pq}-\sum_{k,t}^{n}C_{kt}^{q} b_{ck}b_{lt} \Big)f_l=0.$$

$\textbf{Case (3)}$ $x^{*}=f_{b}^{*}$ and $y^{*}=e_{j}^{*}$.
Then by Eq.\eqref{rOp}, we have:
$$\Big(\sum_{i}^{m}\sum_{s}^{n}C_{is}^{b} a_{ij}b_{st}+\sum_{l,s}^{n}C_{ls}^{j} b_{bl}b_{st}+ \sum_{i}^{m}\sum_{s}^{n}C_{is}^{t} a_{ij}b_{bs}\Big)f_t=0.$$

$\textbf{Case (4)}$ $x^{*}=f_{l}^{*}$ and $y^{*}=f_{t}^{*}$.
Then by Eq.\eqref{rOp}, we have:
$$\Big(-\sum_{d}^{m}\sum_{k,s}^{n}C_{ks}^{d} b_{kl}b_{st} + \sum_{j}^{m}\sum_{s}^{n}C_{js}^{l} a_{dj}b_{st}+ \sum_{j}^{m}\sum_{s}^{n}C_{js}^{t} a_{dj}b_{ls}\Big) e_d=0.$$
Therefore, it is easy to see that $r$ is a solution of  the super MYBE in $A$ if and only
if $r$ satisfies \eqref{rOp}.

Thus the first half part of the conclusion holds. It is easy to get the other results.
\end{proof}

\begin{cor}
Let $(A, [\c,\c])$ be a Malcev superalgebra and $r\in A \otimes A$. Assume $r$ is skew-supersymmetric whith $|r|=0$
and there exists a nondegenerate symmetric invariant bilinear form $B$ on $A$. Define a linear
map $\varphi : A\to A^{\ast}$  by $\langle \varphi(x),y\rangle = B(x, y)$ for any $x, y\in A$. Then $r$ is a solution of the super MYBE in $A$ if and only if  $\tilde{r} = r\varphi : A\to A$ is a Rota-Baxter operator associated to
the module $(A, ad)$, that is, $ \tilde{r}$ satisfies the following equation
$$[\tilde{r}(x),\tilde{r}(y)] = \tilde{r}([\tilde{r}(x), y] + [x, \tilde{r}(y)]),\forall x,y\in A.$$
Hence there is a pre-Malcev superalgebra structure on $A$ given by
\begin{equation}
 x\cdot y= [\tilde{r}(x), y],\ \forall \ x,y\in A.
\end{equation}
\end{cor}
\begin{proof}

For any $x, y\in A$, we have
\begin{align*}
\langle \varphi(ad(x)y), z\rangle&=B([x,y],z)=(-1)^{|z|(|x|+|y)}B(z,[x,y])=-(-1)^{|x||y|}B(y,[x,z])\\
&=-(-1)^{|x||y|}\langle \varphi(y),ad(x)z\rangle=\langle ad^{*}(x)\varphi(y),z\rangle.
\end{align*}
Hence $\varphi(ad(x)y)=ad^{*}(x)\varphi(y)$ for any $x, y\in A$.
Let $x^{*}=\varphi(x)$, $y^{*}=\varphi(y)$, then by Theorem \ref{mybe==dualoop}, $r$ is a solution of the super MYBE in $A$ if and only if
$$
[r\varphi(x), r\varphi(y)] = [r(x^*),r(y^*)] = r(ad^{*}r(x^{*})(y^{*}) -(-1)^{|x||y|} ad^{*}r(y^{*})(x^{*})) = r\varphi\big([r\varphi(x),y]+[x,r\varphi(y)]\big).$$
Therefore the conclusion holds.
\end{proof}

 \subsection{Constructions of solutions of the super MYBE in Malcev superalgebras}
 The following result establishes a close relation between super $\mathcal{O}$-operators and the super classical Yang-Baxter equation
in Malcev superalgebras.
\begin{thm}
Let $(A, [\c,\c])$ be a Malcev superalgebra. Let $\rho^*:A\rightarrow gl(V^*)$ be the dual representation
of the representation $\rho: A\rightarrow gl(V)$ of the Malcev
superalgebra $A$. Let $T:V\to A$  be a linear map which can be identified as an
element in $ (A\ltimes_{\rho^*}V^*)\otimes (A\ltimes_{\rho^*}V^*).$  Then $T$ is a super $\mathcal{O}$-operator of $A$ associated to $(V,\rho)$, with $|T|=0$,
if and only if $r = T -\sigma(T)$ is a skew-supersymmetric solution of the  super MYBE
in $A\ltimes_{\rho^*}V^{\ast}.$
\label{mybe==oop}
\end{thm}
\begin{proof}
Let $\{e_1,...,e_m,f_1,...,f_n\}$ be a basis of $A$ as in the proof of Theorem \ref{mybe==dualoop}. Let $\{u_1,...,u_s,v_1,...,v_t\}$ be a basis of $V$
and $\{u_1 ^{*},...,u_s ^{*},v_1 ^{*},...,v_t^{*}\}$ be the dual basis in a similar way. Since $|T|=0$, we set
$$T(u_i)=\sum_{a=1}^m a_{ia}e_a, i=1,\cdots, s,\quad T(v_k)=\sum_{b=1}^n b_{kb}e_b, k=1,\cdots, t.$$
Moreover, since ${\rm Hom}(V,A)\cong A\otimes V^*,$ we have
\begin{align*}
T&=\sum_{i=1}^s T(u_i)\otimes u_i^* + \sum_{k=1}^t T(v_k)\otimes v_k^*
=\sum_{i=1}^s\sum_{a=1}^m a_{ia}e_a\otimes u_i^* +\sum_{k=1}^t\sum_{b=1}^n b_{kb}e_b \otimes v_k^*\\
&\in A\otimes V^*\subset (A\ltimes_{\rho^*}V^*)\otimes (A\ltimes_{\rho^*}V^*).
\end{align*}
Therefore,
$$r = T -\sigma(T) =\sum_{i=1}^s T(u_i)\otimes u_i^* + \sum_{k=1}^t T(v_k)\otimes v_k^*-\sum_{i=1}^s u^{\ast}_{i}\otimes T(u^{i}) -\sum_{k=1}^t v_k^* \otimes T(v_k).$$
Thus
{\small
\begin{eqnarray*}
 [r_{12},r_{13}] &=&\sum_{i,p=1}^s
\{[T(u_i), T(u_p)]\otimes u_i^*\otimes u_p* -\rho^*
(T(u_i))u_p^*\otimes u_i^*\otimes T(u_p)\\
&& +\rho^*
(T(u_p))u_i^*\otimes
T(u_i)\otimes u_p^* \}\\
&&+\sum_{i=1}^{s}\sum_{q=1}^{t}
\{[T(u_i), T(v_q)]\otimes u_i^*\otimes v_q* -\rho^*
(T(u_i))v_q^*\otimes u_i^*\otimes T(v_q)\\
&&+ \rho^*
(T(v_q))u_i^*\otimes
T(u_i)\otimes v_q^* \}\\
&&+\sum_{k,q=1}^{t}
\{[T(v_k), T(v_q)]\otimes v_k^*\otimes v_q* -\rho^*
(T(v_k))v_q^*\otimes v_k^*\otimes T(v_q)\\
&&+ \rho^*
(T(v_q))v_k^*\otimes
T(v_k)\otimes v_q^* \}\\
&&+\sum_{p=1}^{s}\sum_{k=1}^{t}
\{[T(v_k), T(u_p)]\otimes v_k^*\otimes u_p* -\rho^*
(T(v_k))u_p^*\otimes v_k^*\otimes T(u_p)\\
&&+ \rho^*
(T(u_p))v_k^*\otimes
T(v_k)\otimes u_p^*\},
\end{eqnarray*}
\begin{eqnarray*}
[ r_{12},r_{23}] &=&\sum_{i,p=1}^s \{- T(u_i)\otimes \rho^* (T(u_p))u_i^*\otimes u_p^*- u_i^*\otimes [T(u_i), T(u_p)]\otimes
u_p^*\\
&&+ u_i^*\otimes \rho^*
(T(u_i))u_p^*\otimes
T(u_p)\}\\
&&+\sum_{i=1}^{s}\sum_{q=1}^t \{- T(u_i)\otimes \rho^* (T(v_q))u_i^*\otimes v_q^*- u_i^*\otimes [T(u_i), T(v_q)]\otimes
v_q^*\\
&&+ u_i^*\otimes \rho^*
(T(u_i))v_q^*\otimes
T(v_q)\}\\
&&+\sum_{k,q=1}^t \{- T(v_k)\otimes \rho^* (T(v_q))v_k^*\otimes v_q^*- v_k^*\otimes [T(v_k), T(v_q)]\otimes
v_q^*\\
&&+ v_k^*\otimes \rho^*
(T(v_k))v_q^*\otimes
T(v_q)\}\\
&&+\sum_{p=1}^{s}\sum_{k=1}^t \{- T(v_k)\otimes \rho^* (T(u_p))v_k^*\otimes u_p^*- v_k^*\otimes [T(v_k), T(u_p)]\otimes
u_p^*\\
&&+ v_k^*\otimes \rho^*
(T(v_k))u_p^*\otimes
T(u_p)\},
\end{eqnarray*}
\begin{eqnarray*}
[r_{13},r_{23}] &=&\sum_{i,p=1}^s \{ T(u_i)\otimes u_p^*\otimes \rho^*(T(u_p))u_i^*- u_i^*\otimes T(u_p)\otimes
\rho^* (T(u_i))u_p^*\\
&&+ u_i^*\otimes u_p^*\otimes [T(u_i),
T(u_p)]\}\\
&&+\sum_{i=1}^{s}\sum_{q=1}^t \{ T(u_i)\otimes v_q^*\otimes \rho^*(T(v_q))u_i^*- u_i^*\otimes T(v_q)\otimes
\rho^* (T(u_i))v_q^*\\
&&+ u_i^*\otimes v_q^*\otimes [T(u_i),
T(v_q)]\}\\
&&+\sum_{k,q=1}^t \{ T(v_k)\otimes v_q^*\otimes \rho^*(T(v_q))v_k^*- v_k^*\otimes T(v_q)\otimes
\rho^* (T(v_k))v_q^*\\
&&+ v_k^*\otimes v_q^*\otimes [T(v_k),
T(v_q)]\}\\
&&+\sum_{p=1}^{s}\sum_{k=1}^t \{ T(v_k)\otimes u_p^*\otimes \rho^*(T(u_p))v_k^*- v_k^*\otimes T(u_p)\otimes
\rho^* (T(v_k))u_p^*\\
&&+ v_k^*\otimes u_p^*\otimes [T(v_k),
T(u_p)]
\}.
\end{eqnarray*}
}
On the other hand, from the definition of dual representation, we know
$$\rho^*(T(u_p))u_i^*=-\sum_{j=1}^s u_i^*(\rho(T(u_p))u_j) u_j^*.$$
Thus
\begin{eqnarray*}
-\sum_{i,p=1}^s T(u_i)\otimes \rho^*(T(u_p))u_i^*\otimes
u_p^*&=&\sum_{i,p=1}^s T(u_i)\otimes
[\sum_{j=1}^s u_i^*(\rho(T(u_p))u_j) u_j^*]\otimes u_p^*\\
&=&\sum_{i,p=1}^s \sum_{j=1}^s u_j^*(\rho(T(u_p))u_i) T(u_j)\otimes
u_i^*\otimes u_p^*\\
&=&\sum_{i,p=1}^s T(\sum_{j=1}^s (u_j^*(\rho(T(u_p))u_i) u_j))\otimes
u_i^*\otimes u_p^* \\
&=&\sum_{i,p=1}^s T(\rho (T(u_p))u_i)\otimes u_i^*\otimes u_p^*.
\end{eqnarray*}
Therefore
\begin{eqnarray*}
&&[r_{12},r_{13}]+[r_{12},r_{23}]+[r_{13},r_{23}]\\
&&=\sum_{i,p=1}^s\{
([T(u_i),T(u_p)]+T(\rho(T(u_p))u_i)-T(\rho(T(u_i))u_p))\otimes u_i^*\otimes
u_p^*\\
&&\quad -u_i^*\otimes ([T(u_i),T(u_p)]+T(\rho(T(u_p))u_i)-T(\rho(T(u_i))u_p))\otimes
u_p^*\\
&&\quad + u_i^*\otimes u_p^*\otimes
([T(u_i),T(u_p)]+T(\rho(T(u_p))u_i)-T(\rho(T(u_i))u_p))\}\\
&&\quad+\sum_{i=1}^s\sum_{q=1}^{t}\{
([T(u_i),T(v_q)]+T(\rho(T(v_q))u_i)-T(\rho(T(u_i))v_q))\otimes u_i^*\otimes
v_q^*\\
&&\quad +u_i^*\otimes (-[T(u_i),T(v_q)]- T(\rho(T(v_q))u_i)+T(\rho(T(u_i))v_q))\otimes
v_q^*\\
&&\quad - u_i^*\otimes v_q^*\otimes
(-[T(u_i),T(v_q)]-T(\rho(T(v_q))u_i)+T(\rho(T(u_i))v_q))\}\\
&&\quad+\sum_{k,q=1}^{t}\{
([T(v_k),T(v_q)]+T(\rho(T(v_q))v_k)-T(\rho(T(v_k))v_q))\otimes v_k^*\otimes
v_q^*\\
&& \quad+v_k^*\otimes (-[T(v_k),T(v_q)]- T(\rho(T(v_q))v_k)+T(\rho(T(v_k))v_q))\otimes
v_q^*\\
&&\quad - v_k^*\otimes v_q^*\otimes
(-[T(v_k),T(v_q)]+T(\rho(T(v_k))v_q)-T(\rho(T(v_q))v_k))\}\\
&&\quad+\sum_{p=1}^s\sum_{k=1}^{t}\{
([T(v_k),T(u_p)]+T(\rho(T(u_p))v_k)-T(\rho(T(v_k))u_p))\otimes v_k^*\otimes
u_p^*\\
&&\quad -v_k^*\otimes ([T(v_k),T(u_p)]+T(\rho(T(u_p))v_k)-T(\rho(T(v_k))u_p))\otimes
u_p^*\\
&&\quad + v_k^*\otimes u_p^*\otimes
([T(v_k),T(u_p)]+T(\rho(T(u_p))v_k)-T(\rho(T(v_k))u_p))\}.
\end{eqnarray*}
Obviously, $r$ is a solution of super MYBE in  $A\ltimes_{\rho^*}V^*$
if and only if $T$ is a super $\mathcal{O}$-operator associated to $(V, \rho)$.
\end{proof}

\begin{prop}
Let $(A, [\c,\c])$ be a Malcev superalgebra and $r\in A \otimes A$ be  skew-supersymmetric and nondegenerate with $|r|=0$. Then $r$ is
a  solution of  the super MYBE in $A$ if and only if the bilinear form $\omega$ on $A$ given by
\begin{equation}
\omega(x,y)= \langle r^{-1}(x), y\rangle, \forall x, y\in A,
\label{Rsymplectic}
\end{equation}
is a symplectic form.
\end{prop}
\begin{proof}
Since $r$ is nondegenerate, for $x, y, z\in A$, there exist $a^*
,b^*, c^*\in A^*$ such that $x = r(a^*), y =
r(b^*), z = r(c^*).$

In the one hand, using the fact that $r$ is skew-supersymmetric, then
\begin{eqnarray*}
   \omega(x,y)&=& \langle r^{-1}(x), y\rangle
   =\langle a^*, r(b^*)\rangle
   =-(-1)^{|a^*||b^*|} \langle b^*, r(a^*)\rangle\\
   &=& -(-1)^{|x||y|} \langle r^{-1}(y), x\rangle
   =-(-1)^{|x||y|} \omega(y,x).
\end{eqnarray*}
On the other hand, by Eq.\eqref{eq:2-cocycle} and the definition of the dual representation $ad^*$ of the adjoint representation, we
obtain
\begin{eqnarray*}
&&(-1)^{|x||z|}\omega(x, [ y, z]) + (-1)^{|y||x|}\omega(y, [ z, x])+(-1)^{|z||y|}\omega(z, [ x, y])\\
 &=& (-1)^{|x||z|} \langle r^{-1}(x), [ y, z]\rangle+ (-1)^{|y||x|}\langle r^{-1}(y), [ z, x]\rangle +(-1)^{|z||y|}\langle r^{-1}(z), [ x, y]\rangle\\
&=& (-1)^{|a^*||c^*|} \langle a^*, [  r(b^*), r(c^*)]\rangle+ (-1)^{|b^*||a^*|}\langle b^*, [ r(c^*), r(a^*)]\rangle +(-1)^{|c^*||b^*|}\langle c^*, [ r(a^*), r(b^*)]\rangle\\
&=& -(-1)^{|a^*|(|c^*|+|b^*|}\langle ad^*(r(b^*))a^*, r(c^*)\rangle + (-1)^{|c^*||a^*|}\langle ad^*(r(a^*))b^*,  r(c^*)\rangle +(-1)^{|c^*||b^*|}\langle c^*, [ r(a^*), r(b^*)]\rangle\\
&=& (-1)^{(|a^*|+|c^*|)|b^*|}\langle c^*, r(ad^*(r(b^*))a^*)\rangle - (-1)^{|c^*||b^*|}\langle c^*,  r(ad^*(r(a^*))b^*)\rangle +(-1)^{|c^*||b^*|}\langle c^*, [ r(a^*), r(b^*)]\rangle\\
&=& (-1)^{|c^*||b^*|}\langle c^*, (-1)^{|a^*||b^*|}r(ad^*(r(b^*))a^*)-r(ad^*(r(a^*))b^*)+[ r(a^*), r(b^*)]\rangle.
   \end{eqnarray*}
Hence $\omega$ is a symplectic form if and only if $r$ is
a  solution of  the super MYBE in $A$.
\end{proof}
\begin{cor}
Let $(A,\cdot)$ be a pre-Malcev superalgebra. Let $\{e_1,...,e_m,f_1,...,f_n\}$ be a basis of $A$ where $e_i \in A_{\bar{0}}$ and $f_j\in A_{\bar{1}}$ and let
$\{e_1^*,...,e_m^*,f_1^*,...,f_n^*\}$ be its dual basis with $e^{*} _{l}\in A^{*} _{\bar{0}}$
and $f^{*} _{k}
\in A^{*} _{\bar{1}}$. Then
\begin{equation}
    r=\sum_{i=1}^{m}(e_i\otimes e^{*} _{i}- e^{*} _{i}\otimes e_i )+\sum_{j=1}^{n} ( f_j\otimes f_j^{*}+ f_j^{*} \otimes f_j),
\end{equation}
is a  solution of the super MYBE in $A^{C}\ltimes_{ad^*}(A^{C})^{\ast}$.
\end{cor}
\begin{proof}
It follows from the fact that $T = id$ is a super $\mathcal{O}$-operator of $A^{C}$ associated to the representation
$(A, ad)$.
\end{proof}


\end{document}